\documentclass[12pt,a4paper,reqno,oneside]{amsart}
 
\usepackage[utf8]{inputenc}
\usepackage[T1]{fontenc}
\usepackage{lmodern}
\usepackage{amsmath}
\usepackage{amssymb}
\usepackage{amsthm}
\allowdisplaybreaks
\usepackage{geometry}
\usepackage{graphicx}
\usepackage{bbm}
\usepackage{float}
\usepackage{enumerate}
\usepackage{scrhack}
\usepackage{hyperref}
\usepackage{cleveref}

\usepackage{fancyhdr}

\newtheoremstyle{theorem}	
   {}       			
   {}      				
   {\itshape}  				
   {\parindent} 			
   {\bfseries} 				
   {}         				
   {.7em}      				
   {}          				

\newtheoremstyle{definition}
  {}       				
  {}      				
  {} 		 				
  {\parindent} 				
  {\bfseries} 				
  {}         				
  {.7em}      				
  {}          				
  
\newtheoremstyle{remark}
  {}       			
  {}      			
  {} 					
  {\parindent} 			
  {\itshape} 			
  {.}         			
  {.7em}      			
  {}          			

\theoremstyle{theorem}
\newtheorem{theorem}{Theorem}[section]
\newtheorem{corollary}[theorem]{Corollary}
\newtheorem{lemma}[theorem]{Lemma}
\newtheorem{proposition}[theorem]{Proposition}

\theoremstyle{remark}
\newtheorem{remark}[theorem]{Remark}

\theoremstyle{definition}

\newtheorem{definition}[theorem]{Definition}

\newcommand{\Ebb}{\mathbb{E}}
\newcommand{\Fbb}{\mathbb{F}}

\newcommand{\Nbb}{\mathbb{N}}
\newcommand{\Pbb}{\mathbb{P}}
\newcommand{\Qbb}{\mathbb{Q}}
\newcommand{\Rbb}{\mathbb{R}}

\newcommand{\Bcal}{\mathcal{B}}
\newcommand{\Ccal}{\mathcal{C}}
\newcommand{\Ecal}{\mathcal{E}}
\newcommand{\Fcal}{\mathcal{F}}

\newcommand{\Kcal}{\mathcal{K}}

\newcommand{\Pcal}{\mathcal{P}}

\newcommand{\Lip}{\text{Lip}}

\newcommand{\Xo}{\overline{X}}

\DeclareMathOperator*{\conv}{conv}

\DeclareMathOperator*{\esssup}{ess\,sup}

\newcommand{\btilde}{\tilde{b}}
\newcommand{\bhat}{\hat{b}}

\newcommand{\Ep}[2]{\Ebb\negthinspace\left[\left\Vert #1 \right\Vert^{#2} \right]^{\frac{1}{#2}}\negthinspace}
\newcommand{\EpO}[2]{\Ebb\negthinspace\left[\left\Vert #1 \right\Vert^{#2} \right]\negthinspace}

\newcommand{\Eabs}[1]{\Ebb\negthinspace\left[\left\Vert #1 \right\Vert \right]\negthinspace}
\newcommand{\EW}[1]{\Ebb\negthinspace\left[ #1 \right]\negthinspace}

\usepackage{color}
\definecolor{darkgreen}{rgb}{0, .5, 0}
\definecolor{darkred}{rgb}{.5, 0, 0}

\begin{document}


\title[Multi-Dimensional MFSDEs with Irregular Drift]{Existence and Regularity of Solutions to Multi-Dimensional Mean-Field Stochastic Differential Equations with Irregular Drift}
\author[M.Bauer]{Martin Bauer}
\address{M. Bauer: Department of Mathematics, LMU, Theresienstr. 39, D-80333 Munich, Germany.}
\email{bauer@math.lmu.de}
\author[T. Meyer-Brandis]{Thilo Meyer-Brandis}
\address{T. Meyer-Brandis: Department of Mathematics, LMU, Theresienstr. 39, D-80333 Munich, Germany.}
\email{meyerbra@math.lmu.de}
\date{\today}
\maketitle


\begin{center}
\parbox{13cm}{
\begin{footnotesize}
\textbf{\textsc{Abstract.}} We examine existence and uniqueness of strong solutions of multi-dimensional mean-field stochastic differential equations with irregular drift coefficients. Furthermore, we establish Malliavin differentiability of the solution and show regularity properties such as Sobolev differentiability in the initial data as well as H\"older continuity in time and the initial data. Using the Malliavin and Sobolev differentiability we formulate a Bismut-Elworthy-Li type formula for mean-field stochastic differential equations, i.e. a probabilistic representation of the first order derivative of an expectation functional with respect to the initial condition.\\[0.2cm]
\end{footnotesize}
}
\end{center}

\section{Introduction}
	Let $(\Omega, \Fcal, \Fbb, \Pbb)$ be a complete filtered probability space. Throughout the manuscript let $T>0$ be a finite time horizon. Consider the mean-field stochastic differential equation, hereafter for short mean-field SDE,
	\begin{align}\label{eq:MFSDEGeneral}
		dX_t^x = b \left(t,X_t^x, \Pbb_{X_t^x} \right)dt + \sigma \left( t, X_t^x, \Pbb_{X_t^x} \right) dB_t, \quad t\in [0,T], \quad X_0^x = x \in \Rbb^d,
	\end{align}
	where $b: [0,T] \times \Rbb^d \times \Pcal_1(\Rbb^d) \to \Rbb^d$ is the drift coefficient, $\sigma: [0,T] \times \Rbb^d \times \Pcal_1(\Rbb^d) \to \Rbb^{d\times n}$ the diffusion coefficient, and $\Pbb_{X^x_t} \in \Pcal_1(\Rbb^d)$ denotes the law of $X^x_t$ with respect to the measure $\Pbb$. Here, $B = (B_t)_{t\in [0,T]}$ is $n$-dimensional Brownian motion and $\Pcal_1(\Rbb^d)$ is the space of probability measures over $(\Rbb^d, \Bcal(\Rbb^d))$ with finite first moment. \par
	Mean-field SDE \eqref{eq:MFSDEGeneral}, also called McKean-Vlasov equation, originates in the study on multi-particle systems with weak interaction and traces back to works of Vlasov \cite{Vlasov_VibrationalPropertiesofElectronGas}, Kac \cite{Kac_FoundationsOfKineticTheory}, and McKean \cite{McKean_AClassOfMarkovProcess}. In recent years the interest in mean-field SDEs increased due to the work of Lasry and Lions \cite{LasryLions_MeanFieldGames} on mean-field games and the related application in the fields of Economics and Finance, for example in the study of systemic risk, see e.g. \cite{CarmonaFouqueMousaviSun_SystemicRiskandStochasticGameswithDelay}, \cite{CarmonaFouqueSun_MFGandSystemicRisk}, \cite{FouqueIchiba_StabilityinaModelofInterbankLending}, \cite{FouqueSun_SystemicRiskIllustrated}, \cite{GarnierPapanicolaouYang_LargeDeviationsforMFModelofSystemicRisk}, \cite{KleyKlueppelbergReichel_SystemicRiskTroughContagioninaCorePeripheryStructuredBankingNetwork}, and the cited sources therein. Carmona and Delarue developed subsequently the theory on mean-field games in a mere probabilistic environment, cf. \cite{CarmonaDelarue_ProbabilisticAnalysisofMFG}, \cite{CarmonaDelarue_MasterEquation}, \cite{CarmonaDelarue_FBSDEsandControlledMKV}, \cite{CarmonaDelarue_Book}, \cite{CarmonaDelarueLachapelle_ControlofMKVvsMFG}, and \cite{CarmonaLacker_ProbabilisticWeakFormulationofMFGandApplications}.

	In this paper the focus lies on existence and uniqueness as well as regularity properties of solutions to multi-dimensional mean-field SDEs with additive noise, i.e. equations of the form
	\begin{align}\label{eq:MFSDEAdditive}
		dX_t^x = b \left(t,X_t^x, \Pbb_{X_t^x} \right)dt + dB_t, \quad t\in [0,T], \quad X_0^x = x \in \Rbb^d,
	\end{align}
	where $B$ is $d$-dimensional Brownian motion. In particular, we are interested in irregular drift coefficients $b$ that are merely measurable in the spatial variable.

	Existence and uniqueness of solutions to mean-field SDEs have been discussed in several works, cf. for example \cite{Bauer_StrongSolutionsOfMFSDEs}, \cite{BuckdahnDjehicheLiPeng_MFBSDELimitApproach}, \cite{BuckdahnLiPeng_MFBSDEandRelatedPDEs}, \cite{BuckdahnLiPengRainer_MFSDEandAssociatedPDE}, \cite{Chiang_MKVWithDiscontinuousCoefficients}, \cite{de2015strong}, \cite{JourdainMeleardWojbor_NonlinearSDEs}, \cite{LiMin_WeakSolutions}, \cite{mahmudov2006mckean}, and \cite{MishuraVeretennikov_SolutionsOfMKV}. Li and Min show in \cite{LiMin_WeakSolutions} the existence of a weak solution for a path dependent mean-field SDE, where the drift $b$ is assumed to be bounded and continuous in the law variable. Under the additional assumption that $b$ admits a modulus of continuity they prove uniqueness in law of the solution. In \cite{MishuraVeretennikov_SolutionsOfMKV}, the authors derive existence of a pathwisely unique strong solution for drift coefficients $b$ of at most linear growth that are continuous in the law variable with respect to the total variation metric. In order to prove their result, Mishura and Veretennikov use an approach similar to Krylov in his analysis of stochastic differential equations, cf. \cite{krylov1969ito}. The one-dimensional case of mean-field SDE \eqref{eq:MFSDEAdditive} is considered in \cite{Bauer_StrongSolutionsOfMFSDEs}. There, we show that mean-field SDE \eqref{eq:MFSDEAdditive} has a Malliavin differentiable pathwisely unique strong solution for drift coefficients $b$ admitting a modulus of continuity in the law variable and having a decomposition 
	\begin{align}\label{eq:decomp}
		b(t,y,\mu) := \bhat(t,y,\mu) + \btilde(t,y,\mu),
	\end{align}
	where $\bhat$ is merely measurable and bounded and $\btilde$ is of at most linear growth and Lipschitz continuous in the spatial variable. We remark that in \cite{Bauer_StrongSolutionsOfMFSDEs} the decomposition \eqref{eq:decomp} is required to establish regularity properties such as Malliavin differentiability of the strong solution, whereas for mere existence of a strong solution it suffices to assume the drift coefficient to be of at most linear growth and continuous in the law variable, see also \Cref{thm:strongSolution} below.\par
	Regularity properties of solutions to mean-field SDEs are investigated for example in \cite{Bauer_StrongSolutionsOfMFSDEs}, \cite{BuckdahnLiPengRainer_MFSDEandAssociatedPDE}, and \cite{crisan2018smoothing}. In \cite{BuckdahnLiPengRainer_MFSDEandAssociatedPDE} and \cite{crisan2018smoothing}, the authors derive Malliavin differentiability of solutions to mean-field SDE \eqref{eq:MFSDEGeneral} for regular coefficients $b$ and $\sigma$. Further, they examine in the case of regular coefficients differentiability of the solution with respect to the initial value. In their analysis they use the notion of Lions derivative which denotes the derivative with respect to a measure. We derive in \cite{Bauer_StrongSolutionsOfMFSDEs} Malliavin differentiability, Sobolev differentiability in the initial data, and Hölder continuity in time and initial data for the one-dimensional mean-field SDE \eqref{eq:MFSDEAdditive} but for drift coefficients that are merely Lipschitz continuous in the law variable and admit a decomposition \eqref{eq:decomp}. In particular, we prove Sobolev differentiability in the initial data without using the notion of Lions derivative. Lastly, we show that the expectation functional $\Ebb[(\Phi(X_T^x)]$ is Sobolev differentiable with respect to $x$, where $X^x$ is the unique strong solution of mean-field SDE \eqref{eq:MFSDEAdditive} and $\Phi: \Rbb \to \Rbb$ satisfies merely some integrability condition. Further, we derive a Bismut-Elworthy-Li type formula for the derivative $\nabla_x \Ebb[(\Phi(X_T^x)]$.\footnote{Here, $\nabla_x$ denotes the Jacobian with respect to the variable $x$.}\par

	The main objective of this paper is to extend the results obtained in \cite{Bauer_StrongSolutionsOfMFSDEs} to the multi-dimensional case. More precisely, at first we show existence of a strong solution for drift coefficients $b$ that are merely measurable, of at most linear growth, and continuous in the law variable. Here, we proceed as in \cite{Bauer_StrongSolutionsOfMFSDEs} to show first existence of a weak solution by applying Girsanov's theorem and Schauder's fixed point theorem, and then resort to existence results of SDE's to guarantee the existence of a strong solution. Under the additional assumption that $b$ admits a modulus of continuity in the law variable pathwise uniqueness of the solution is derived. If the drift coefficient $b$ is bounded and continuous in the law variable, we further show that the strong solution of the multi-dimensional mean-field SDE \eqref{eq:MFSDEAdditive} is Malliavin differentiable. Finally, for $b$ being merely bounded and Lipschitz continuous in the law variable, Sobolev differentiability in the initial data and Hölder continuity in time and intitial data as well as a Bismut-Elworthy-Li type formula are derived. \par	
		The main difference compared to the one-dimensional case in \cite{Bauer_StrongSolutionsOfMFSDEs} in the courses of the proofs of Sobolev differentiability, Hölder continuity, and the Bismut-Elworthy-Li formula is that there does not exist a representation of the Malliavin derivative by means of integration with respect to local time. Instead, we derive in a first step for regular drift coefficients $b$ the relation 
	\begin{align*}
		\nabla_x X_t^x = D_sX_t^x \nabla_x X_s^x + \int_s^t D_r X_t^x \nabla_x b(r, y, \Pbb_{X_r^x})\big\vert_{y= X_r^x} dr, ~0 \leq s \leq t \leq T,
	\end{align*}
		where $(D_sX_t^x)_{0\leq s \leq t \leq T}$ is the Malliavin derivative and $(\nabla_x X_t^x)_{0\leq t \leq T}$ the Sobolev derivative of the strong solution $X^x$ of mean-field SDE \eqref{eq:MFSDEAdditive}. Afterwards we use this relation to derive the pursued regularity properties for irregular drift coefficients $b$ by applying an approximational approach. \par	

	The paper is structured as follows. In \Cref{sec:assumptions} we give the definitions of the assumptions applied on the drift function $b$. \Cref{sec:solution} contains the main result on existence of a pathwisely unique solution. Afterwards, we discuss the properties of Malliavin and Sobolev differentiability as well as Hölder continuity in \Cref{sec:Malliavin,sec:Sobolev,sec:Hölder}, respectively. The paper is closed by deriving a Bismut-Elworthy-Li type formula in \Cref{sec:Bismut}.

\section{Notation and Assumptions} \label{sec:assumptions}
	Subsequently we list some of the most frequently used notations.
\begin{itemize}
\item $\lbrace e_k \rbrace_{1 \leq k \leq d}$ is the standard basis of $\Rbb^d$ consisting of the unit vectors.
\item $\Ccal_b^{1,1}(\Rbb^d)$ is the space of continuously differentiable functions $f:\Rbb^d \to \Rbb^d$ with bounded and Lipschitz continuous partial derivatives.
\item $\Ccal_0^\infty(\Rbb^d)$ denotes the space of smooth functions with compact support. 
\item $L^\infty \left([0,T], \Ccal_b^{1,L}\left(\Rbb^d \times \Pcal_1\left(\Rbb^d \right) \right) \right)$ is the space of functions $f:[0,T] \times \Rbb^d \times \Pcal_1 \left( \Rbb^d \right) \to \Rbb^d$ such that 
\begin{itemize}
	\item $t \mapsto f(t,y,\mu)$ is bounded uniformly in $y\in \Rbb$ and $\mu \in \Pcal_1\left( \Rbb^d \right)$
	\item $\left( y \mapsto f(t,y,\mu)\right) \in \Ccal_b^{1,1}(\Rbb^d)$ uniformly in $t\in [0,T]$ and $\mu \in \Pcal_1\left( \Rbb^d \right)$
	\item $\mu \mapsto f(t,y,\mu)$ is Lipschitz continuous uniformly in $t\in [0,T]$ and $y \in \Rbb^d$.
\end{itemize}
\item $\delta_0$ denotes the Dirac measure in $0$.
\item $\Lip_1\left(\Rbb^d, \Rbb \right)$ denotes the set of functions $f: \Rbb^d \to \Rbb$ that are Lipschitz continuous with Lipschitz constant $1$.
\item The Kantorovich metric on the space $\Pcal_1(\Rbb^d)$ is defined by
	\begin{align*}
		\Kcal(\mu, \nu) := \sup_{h \in \Lip_1(\Rbb^d, \Rbb)} \left\vert \int_{\Rbb^d} h(y) (\mu - \nu)(dy) \right\vert, \quad \mu, \nu \in \Pcal_1 \left(\Rbb^d \right).
	\end{align*}
\item We write $E_1(\theta) \lesssim E_2(\theta)$ for two mathematical expressions $E_1(\theta),E_2(\theta)$ depending on some parameter $\theta$, if there exists a constant $C>0$ not depending on $\theta$ such that $E_1(\theta) \leq C E_2(\theta)$.
\item $\Vert \cdot \Vert_\infty$ sup norm over all variables
\item $\Vert \cdot \Vert$ is the euclidean norm
\item $\nabla_x$ is the Jacobian in the direction of the variable $x \Rbb^d$, $\nabla_k$ is the Jacobian in the direction of the $k$-th variable, $\partial_x$ is the (weak) partial derivative in the direction of the variable $x \in \Rbb$, $\partial_k$ is the (weak) partial derivative in the direction of $e_k$.
\item We define the weight function
	\begin{align}\label{eq:weightFunction}
		\omega_T(y) := \exp \left\lbrace - \frac{\Vert y\Vert^2}{4T} \right\rbrace, \quad y \in \Rbb^d,
	\end{align}
	and the weighted $L^2$-space $L^2(\Rbb^d; \omega_T)$ as the space of functions $f:\Rbb^d \to \Rbb^d$ such that 
	\begin{align*}
		\left( \int_{\Rbb^d} \Vert f(y) \Vert^2 \omega_T(y) dy \right)^\frac{1}{2} < \infty.
	\end{align*}
\end{itemize}

	In the following we give conditions on the drift function $${b: [0,T] \times \Rbb^d \times \Pcal_1(\Rbb^d) \to \Rbb^d}$$ that we use frequently throughout the paper. \bigskip
	
	We say that the function $b$ is of \emph{linear growth}, if there exists a constant $C>0$ such that for every $t\in [0,T]$, $y \in \Rbb^d$, and $\mu \in \Pcal_1(\Rbb^d)$
	\begin{align}\label{eq:linearGrowth}
		\Vert b(t,y, \mu) \Vert \leq C \left( 1 + \Vert y \Vert + \Kcal(\mu, \delta_0) \right).
	\end{align}
	
	The function $b$ is said to be \emph{continuous in the third variable} (uniformly with respect to the first and second variable), if for every $\mu \in \Pcal_1(\Rbb^d)$ and $\varepsilon >0$ there exists $\delta >0$ such that for all $\nu \in \Pcal_1(\Rbb^d)$ with $\Kcal(\mu, \nu) < \delta$, we have for all $t\in [0,T]$ and $y\in \Rbb^d$
	\begin{align}\label{eq:continuousThird}
		\Vert b(t,y, \mu) - b(t,y,\nu) \Vert < \varepsilon.
	\end{align}

	The drift coefficient $b$ admits a \emph{modulus of continuity} (in the third variable), if  there exists a continuous function $\theta: \Rbb_+ \to \Rbb_+$ with $\int_0^z (\theta(y))^{-1} dy = \infty$ for all $z \in \Rbb_+$ such that for every $t \in [0,T]$, $y \in \Rbb^d$, and $\mu, \nu \in \Pcal_1(\Rbb^d)$
	\begin{align}\label{eq:modulusOfContinuity}
		\Vert b(t,y,\mu) - b(t,y,\nu) \Vert^2 \leq \theta\left( \Kcal(\mu, \nu)^2 \right).
	\end{align}
	
	We say the drift coefficient $b$ is \emph{Lipschitz continuous in the third variable} (uniformly with respect to the first and second variable), if there exists a constant $C>0$ such that for all $t\in [0,T]$, $y \in \Rbb^d$, and $\mu, \nu \in \Pcal_1(\Rbb^d)$
	\begin{align}\label{eq:Lipschitz}
		\Vert b(t,y,\mu) - b(t,y,\nu) \Vert \leq C \Kcal(\mu, \nu).
	\end{align}

\section{Existence and Uniqueness of Solutions}\label{sec:solution}
	In this section we investigate under which of the assumptions specified in \Cref{sec:assumptions} on the drift coefficient $b$ mean-field SDE \eqref{eq:MFSDEAdditive} has a (strong) solution and moreover, in which case this solution is unique. Let us recall the definitions of weak and strong solutions as well as weak and pathwise uniqueness.
	
\begin{definition}[Weak Solution]
	A six-tuple ${(\Omega, \Fcal, \Fbb, \Pbb, B, X^x)}$ is called \emph{weak solution} of mean-field SDE \eqref{eq:MFSDEAdditive}, if
	\begin{enumerate}[(i)]
	\item $(\Omega, \Fcal, \Fbb, \Pbb)$ is a complete filtered probability space and $\Fbb = \lbrace \Fcal_t \rbrace_{t\in [0,T]}$ satisfies the usual conditions of right-continuity and completeness,
	\item $B=(B_t)_{t\in [0,T]}$ is $d$-dimensional $(\Fbb, \Pbb)$-Brownian motion,
	\item $X^x = (X^x_t)_{t\in [0,T]}$ is an a.s.~continuous, $\Fbb$-adapted, $\Rbb^d$-valued process which satisfies $\Pbb$-a.s. equation \eqref{eq:MFSDEAdditive}.
	\end{enumerate}
\end{definition}

\begin{definition}[Strong Solution]
	A \emph{strong solution} of mean-field SDE \eqref{eq:MFSDEAdditive} is a weak solution $(\Omega, \Fcal, \Fbb^B, \Pbb, B, X^x)$ where $\Fbb^B$ is the filtration generated by the Brownian motion $B$ and augmented with the $\Pbb$-null sets.
\end{definition}

\begin{remark}
	In the following we merely speak of $X^x$ as a weak and a strong solution of mean-field SDE \eqref{eq:MFSDEAdditive}, respectively, if there is no ambiguity concerning the stochastic basis $(\Omega, \Fcal, \Fbb, \Pbb, B)$.
\end{remark}

\begin{definition}[Uniqueness in Law]
	A weak solution $(\Omega, \Fcal, \Fbb, \Pbb, B, X^x)$ of mean-field SDE \eqref{eq:MFSDEAdditive} is said to be \emph{weakly unique} or \emph{unique in law}, if for any other weak solution $(\widetilde{\Omega}, \widetilde{\Fcal}, \widetilde{\Fbb}, \widetilde{\Pbb}, \widetilde{B}, Y^x)$ of \eqref{eq:MFSDEAdditive} with the same initial condition $X_0^x = Y_0^x$, it holds that
	\begin{align*}
		\Pbb_{X^x} = \widetilde{\Pbb}_{Y^x}.
	\end{align*}
\end{definition}

\begin{definition}[Pathwise Uniqueness]
	A weak solution $(\Omega, \Fcal, \Fbb, \Pbb, B, X^x)$ of mean-field SDE \eqref{eq:MFSDEAdditive} is said to be \emph{pathwisely unique}, if for any other weak solution $Y^x$ with respect to the same stochastic basis $(\Omega, \Fcal, \Fbb, \Pbb, B)$ with the same initial condition $X_0^x = Y_0^x$, it holds that
	\begin{align*}
		\Pbb\left( \forall t\geq 0: X_t^x = Y_t^x \right) = 1.
	\end{align*}
\end{definition}

\begin{remark}
	Since for strong solutions of mean-field SDE's of type \eqref{eq:MFSDEAdditive} the notions of pathwise uniqueness and uniqueness in law are equivalent (cf. \cite[Remark 2.11]{Bauer_StrongSolutionsOfMFSDEs}), we merely speak of a unique strong solution, if a strong solution is unique in any of the two senses.
\end{remark}

The following result provides sufficient conditions allowing for irregular drift coefficients $b$ such that mean-field SDE \eqref{eq:MFSDEAdditive} has a (unique) strong solution. Note that in \cite[Proposition 2]{MishuraVeretennikov_SolutionsOfMKV} a similar result on the existence of a strong solution of mean-field SDE \eqref{eq:MFSDEAdditive} is derived where the authors assume drift coefficients of at most linear growth that are continuous in the law variable with respect to the topology of weak convergence. Here, in contrast to \cite{MishuraVeretennikov_SolutionsOfMKV}, we assume continuity in the law variable merely with respect to the Kantorovich metric and provide a more direct alternative of proof that is not based on approximation arguments.

\begin{theorem}\label{thm:strongSolution}
	Suppose the drift coefficient $b:[0,T] \times \Rbb^d \times \Pcal_1(\Rbb^d) \to \Rbb^d$ is of at most linear growth \eqref{eq:linearGrowth} and continuous in the third variable \eqref{eq:continuousThird}. Then, mean-field SDE \eqref{eq:MFSDEAdditive} has a strong solution. \par
	If in addition $b$ is admitting a modulus of continuity \eqref{eq:modulusOfContinuity}, the solution is unique.
\end{theorem}
\begin{proof}
	First note that identically to \cite[Theorem 2.3]{Bauer_StrongSolutionsOfMFSDEs} one can show that under the assumptions of linear growth \eqref{eq:linearGrowth} and continuity in the third variable \eqref{eq:continuousThird} on the drift coefficient $b$, mean-field SDE \eqref{eq:MFSDEAdditive} has a weak solution $(X_t^x)_{t\in[0,T]}$ for any finite time horizon $T>0$. In particular, $\Pbb_{X^x} \in \Ccal([0,T]; \Pcal_1(\Rbb^d))$ and due to \Cref{lem:boundDrift} for every $p\geq 1$
	\begin{align}\label{eq:LpBound}
		\EW{ \sup_{t\in [0,T]} \Vert X_t^x \Vert^p } < \infty.
	\end{align}
	In order to show the existence of a strong solution, consider the stochastic differential equation 
	\begin{align}\label{eq:auxSDE}
		dY_t^x = b^{\Pbb_{X}}\left(t,Y_t^x \right) dt + dB_t, \quad t\in [0,T], \quad Y_0^x = x \in \Rbb^d,
	\end{align}
	where $b^{\Pbb_{X}}(t,y) := b(t,y,\Pbb_{X_t^x})$ for all $t\in [0,T]$ and $y\in \Rbb^d$. Due to the work of Veretennikov \cite{veretennikov1981strong} it is well-known that SDE \eqref{eq:auxSDE} has a unique strong solution $(Y_t)_{t\in [0, \tau]}$ up to the time of explosion $\tau > 0$. Since $X^x$ is a weak solution of SDE \eqref{eq:auxSDE} on the interval $[0,T]$, both processes $X^x$ and $Y^x$ must coincide on the interval $[0,\tau]$, due to uniqueness of the solution $Y$ to SDE \eqref{eq:auxSDE}. But due to condition \eqref{eq:LpBound}, $X^x$ is almost surely finite on the interval $[0,T]$ and thus $Y^x$ is also almost surely finite on the interval $[0,T]$. Consequently, $Y^x$ is a strong solution of SDE \eqref{eq:auxSDE} on the interval $[0,T]$ which coincides pathwisely and in law with $X^x$. In particular, for all $t \in [0,T]$
	\begin{align*}
		\Pbb_{Y_t} = \Pbb_{X_t},
	\end{align*}
	and thus, SDE \eqref{eq:auxSDE} and mean-field SDE \eqref{eq:MFSDEAdditive} coincide and $Y^x$ is a strong solution of mean-field SDE \eqref{eq:MFSDEAdditive}. \par
	If in addition $b$ admits a modulus of continuity \eqref{eq:modulusOfContinuity}, it can be shown analogously to \cite[Theorem 2.7]{Bauer_StrongSolutionsOfMFSDEs} that the weak solution of mean-field equation \eqref{eq:MFSDEAdditive} is unique in law. This in fact yields a unique associated SDE \eqref{eq:auxSDE}. In addition with the uniqueness of the strong solution to SDE \eqref{eq:auxSDE}, this yields a unique strong solution of mean-field equation \eqref{eq:MFSDEAdditive}.
\end{proof}

\section{Regularity Properties}\label{sec:regularity}

\subsection{Malliavin Differentiability}\label{sec:Malliavin}
	Similar to the existence of a strong solution, the property of being Malliavin differentiable transfers directly from the solution $Y^x$ of SDE \eqref{eq:auxSDE} to the solution $X^x$ of mean-field SDE \eqref{eq:MFSDEAdditive}. Thus, we immediately get from \cite[Theorem 3.3]{MenoukeuMeyerBrandisNilssenProskeZhang_VariationalApproachToTheConstructionOfStrongSolutions} the following result.
	
\begin{theorem}\label{thm:Malliavin}
	Suppose the drift coefficient $b:[0,T] \times \Rbb^d \times \Pcal_1(\Rbb^d) \to \Rbb^d$ is continuous in the third variable \eqref{eq:continuousThird} and bounded. Then, the strong solution $(X_t^x)_{t \in [0,T]}$ of mean-field SDE \eqref{eq:MFSDEAdditive} is Malliavin differentiable.
\end{theorem}

\subsection{Sobolev Differentiability}\label{sec:Sobolev}
	In this section we consider the unique strong solution of mean-field SDE \eqref{eq:MFSDEAdditive} as a function in the initial value $x$, i.e. for every $t \in [0,T]$ we consider the function $x \mapsto X_t^x$. More precisely, we are interested in the existence of the first variation process $(\nabla_x X_t^x)_{t\in[0,T]}$ in a weak (Sobolev) sense. Let us first recall the definition of the Sobolev space $W^{1,2}(U)$ and then state the main result of this section.

\begin{definition}
	Let $U \subset \Rbb^d$ be an open and bounded subset. The Sobolev space $W^{1,2}(U)$ is defined as the set of functions $u:\Rbb^d \to \Rbb^d$, $u\in L^2(U)$, such that its weak derivative belongs to $L^2(U)$. Furthermore, the Sobolev space is endowed with the norm
\begin{align*}
	\Vert u \Vert_{W^{1,2}(U)} = \Vert u \Vert_{L^2(U)} + \sum_{k=1}^d \Vert \partial_k u \Vert_{L^2(U)}.
\end{align*}
We say a stochastic process $X$ is Sobolev differentiable in $U$, if for all $t\in[0,T]$, $X_t^{\cdot}$ belongs $\Pbb$-a.s. to $W^{1,2}(U)$.
\end{definition}
 
\begin{theorem}\label{thm:Sobolev}
	Suppose the drift coefficient $b:[0,T] \times \Rbb^d \times \Pcal_1(\Rbb^d) \to \Rbb^d$ is Lipschitz continuous in the third variable \eqref{eq:Lipschitz} and bounded. Let $(X_t^x)_{t\in[0,T]}$ be the unique strong solution of mean-field SDE \eqref{eq:MFSDEAdditive} and $U\subset \Rbb^d$ be an open and bounded subset. Then, for every $t\in [0,T]$
	\begin{align*}
		\left(x \mapsto X_t^x \right) \in L^2\left(\Omega, W^{1,2}(U)\right).
	\end{align*}
\end{theorem}

The remaining part of this subsection is devoted to the proof of \Cref{thm:Sobolev}. We start by showing that the result does hold for regular drift coefficients $b$. Subsequently, we define a sequence $\lbrace b_n \rbrace_{n\geq 1}$ of regular functions that approximate the irregular drift coefficient $b$ from \Cref{thm:Sobolev} and prove that the strong solutions $\lbrace X^{n,x} \rbrace_{n \geq 1}$ to the corresponding mean-field SDEs converge strongly in $L^2(\Omega)$ to the solution $X^x$ of \eqref{eq:MFSDEAdditive}. Concluding we get by showing that $\lbrace X^{n,x} \rbrace_{n \geq 1}$ is weakly relatively compact in the space $L^2\left(\Omega, W^{1,2}(U)\right)$ that $X^x$ is Sobolev differentiable as a function in the initial value $x$.

\begin{proposition}\label{prop:SobolevReg}
	Let the drift coefficient $b\in L^\infty \left([0,T], \Ccal_b^{1,L}\left(\Rbb^d \times \Pcal_1\left(\Rbb^d\right) \right) \right)$ and let $(X_t^x)_{t\in[0,T]}$ be the unique strong solution of mean-field SDE \eqref{eq:MFSDEAdditive}. Then, for all $t\in [0,T]$ the map $x \mapsto X_t^x$ is a.s. Lipschitz continuous and consequently weakly and almost everywhere differentiable.
\end{proposition}
\begin{proof}
	The proof is equivalent to the proof of \cite[Proposition 3.5]{Bauer_StrongSolutionsOfMFSDEs}.
\end{proof}

\begin{corollary}
	 The map $x \mapsto b(s,y,\Pbb_{X_s^x})$ is Lipschitz continuous for all $t\in[0,T]$ and $y \in \Rbb^d$ under the assumptions of \Cref{prop:SobolevReg} and thus weakly and almost everywhere differentiable. Moreover, for every $0 \leq s < t \leq T$
	\begin{align}\label{eq:RepFVMD}
		\nabla_x X_t^x = D_sX_t^x \nabla_x X_s^x + \int_s^t D_r X_t^x \nabla_x b(r, y, \Pbb_{X_r^x})\big\vert_{y= X_r^x} dr.
	\end{align}
\end{corollary}
\begin{proof}
	Similar to the proof of \cite[Proposition 3.5]{Bauer_StrongSolutionsOfMFSDEs} it can be shown that $x \mapsto b(s,y,\Pbb_{X_s^x})$ is Lipschitz continuous for all $t\in[0,T]$ and $y \in \Rbb^d$. Furthermore, consider the linear affine ODE 
	\begin{align}\label{eq:ODEFV}
		Z_t = I_d + \int_0^t \nabla_2 b\left(s, X_s^x, \Pbb_{X_s^x}\right) Z_s + \nabla_x b\left(s, y, \Pbb_{X_s^x}\right)\big\vert_{y= X_s^x} ds.
	\end{align}
	First note that $\nabla_x X_t^x$ is a solution of ODE \eqref{eq:ODEFV}. Moreover, by assumption $\Vert \nabla_2 b \Vert_\infty \leq C_1 < \infty$ for some constant $C_1>0$ and since $x \mapsto X_s^x$ is Lipschitz continuous for all $s \in [0,T]$ we get
	\begin{align*}
		\left\Vert \nabla_x b\left(s, y, \Pbb_{X_s^x}\right)\big\vert_{y= X_s^x} \right\Vert &\leq \sum_{k=1}^d \lim_{x_0^{(k)} \to x^{(k)}} \left\Vert \frac{b\left(s, X_s^x, \Pbb_{X_s^x} \right) - b\left(s, X_s^x, \Pbb_{X_s^{\overline{x_0}^{(k)}}} \right)}{x^{(k)} - x_0^{(k)}} \right\Vert \\
		&\lesssim \sum_{k=1}^d \lim_{x_0^{(k)} \to x^{(k)}} \frac{\Kcal \left( \Pbb_{X_s^x}, \Pbb_{X_s^{\overline{x_0}^{(k)}}} \right)}{\left\vert x^{(k)} - x_0^{(k)} \right\vert} \lesssim 1,
	\end{align*}
	where $\overline{x_0}^{(k)} = x + \langle x_0 - x, e_k \rangle.$ Therefore, $\Vert  \nabla_x b(s, y, \Pbb_{X_s^x})\vert_{y= X_s^x} \Vert_\infty \leq C_2 < \infty$ for some constant $C_2>0$ and consequently, ODE \eqref{eq:ODEFV} has the unique solution $\nabla_x X_t^x$. On the other hand, the Malliavin derivative $D_sX_t^x$, $0\leq s < t \leq T$, is the unique solution to the homogeneous ODE
	\begin{align*}
		D_s X_t^x = I_d + \int_s^t \nabla_2 b\left(r, X_r^x, \Pbb_{X_r^x}\right) D_s X_r^x dr.
	\end{align*}
	Consequently, we get that the Malliavin derivative has the explicit representation
	\begin{align*}
		D_s X_t^x = \exp\left\lbrace \int_s^t \nabla_2 b\left(r, X_r^x, \Pbb_{X_r^x}\right) dr \right\rbrace,
	\end{align*}
	and the first variation process has the representation
	\begin{align*}
		\nabla_x X_t^x = D_0 X_t^x \left( I_d + \int_0^t \left( D_0 X_r \right)^{-1} \nabla_x b\left(r, y, \Pbb_{X_r^x}\right)\big\vert_{y= X_r^x} dr \right).
	\end{align*}
	Thus, we get
	\begin{align*}
		D_sX_t^x \nabla_x X_s^x &= D_0 X_t^x \left( I_d + \int_0^s \left( D_0 X_r \right)^{-1} \nabla_x b\left(r, y, \Pbb_{X_r^x}\right)\big\vert_{y= X_r^x} dr \right) \\
		&= D_0 X_t^x + \int_0^s D_r X_t \nabla_x b\left(r, y, \Pbb_{X_r^x}\right)\big\vert_{y= X_r^x} dr \\
		&= \nabla_x X_t^x - \int_s^t D_r X_t \nabla_x b\left(r, y, \Pbb_{X_r^x}\right)\big\vert_{y= X_r^x} dr.
	\end{align*}
	Rearranging yields equation \eqref{eq:RepFVMD}.
\end{proof}

	Now consider a general drift coefficient $b$ which fulfills the assumptions of \Cref{thm:Sobolev}, namely Lipschitz continuity in the third variable \eqref{eq:Lipschitz} and boundedness, and let $X^x$ be the corresponding unique strong solution of mean-field SDE \eqref{eq:MFSDEAdditive}. Due to standard approximation arguments there exists a sequence of approximating drift coefficients
	\begin{align}\label{eq:approxDrift}
		b_n \in L^\infty\left([0,T], \Ccal_b^{1,L}\left(( \Rbb^d \times \Pcal_1\left(\Rbb^d\right)\right)\right), \quad n\geq 1,
	\end{align}
	with $\sup_{n\geq 1} \Vert b_n \Vert_\infty \leq C < \infty$ such that $b_n \to b$ pointwise in every $\mu$ and a.e. in $(t,y)$ with respect to the Lebesgue measure. We denote $b_0 := b$ and assume that the drift coefficients $b_n$ are Lipschitz continuous in the third variable \eqref{eq:Lipschitz} uniformly in $n\geq 0$. We define the corresponding mean-field SDEs
	\begin{align}\label{eq:approxSDE}
		dX_t^{n,x} = b_n \left(t, X_t^{n,x}, \Pbb_{X_t^{n,x}} \right)dt + dB_t, \quad t\in [0,T], \quad X_0^{n,x} = x \in \Rbb^d,
	\end{align}
	which admit unique Malliavin differentiable strong solutions due to \Cref{thm:strongSolution} and \Cref{thm:Malliavin}. Moreover, the solutions $\lbrace X^{n,x} \rbrace_{n\geq 1}$ are Sobolev differentiable in the initial condition $x$ by \Cref{prop:SobolevReg}. Subsequently, we show that $(X_t^{n,x})_{t\in [0,T]}$ converges to $(X_t^x)_{t\in[0,T]}$ in $L^2(\Omega, \Fcal_t)$ as $n\to \infty$.

\begin{proposition}\label{prop:L2Convergence}
	Suppose the drift coefficient $b:[0,T] \times \Rbb^d \times \Pcal_1(\Rbb^d) \to \Rbb^d$ is Lipschitz continuous in the third variable \eqref{eq:Lipschitz} and bounded. Let $(X_t^x)_{t\in[0,T]}$ be the unique strong solution of mean-field SDE \eqref{eq:MFSDEAdditive}. Furthermore, $\lbrace b_n \rbrace_{n\geq 1}$ is the approximating sequence as defined in \eqref{eq:approxDrift} and $(X_t^{n,x})_{t\in[0,T]}$, $n\geq 1$, the corresponding unique strong solutions of \eqref{eq:approxSDE}. Then, there exists a subsequence $\lbrace n_k \rbrace_{k\geq 1} \subset \Nbb$ such that
	\begin{align*}
		X_t^{n_k,x} \xrightarrow[k\to \infty]{} X_t^x, \quad t\in [0,T],
	\end{align*}
	strongly in $L^2(\Omega, \Fcal_t)$.
\end{proposition}
\begin{proof}
	In \cite[Corollary 3.6]{MenoukeuMeyerBrandisNilssenProskeZhang_VariationalApproachToTheConstructionOfStrongSolutions} it is shown in the case of SDEs that for every $t\in[0,T]$ the sequence $\lbrace X^{n,x} \rbrace_{n\geq 1}$ is relatively compact in $L^2(\Omega, \Fcal_t)$. Due to \Cref{thm:Malliavin} the proof therein can be extended to the case of mean-field SDEs under the assumptions of \Cref{prop:L2Convergence}. Thus, for every $t\in [0,T]$ we can find a subsequence $\lbrace n_k(t) \rbrace_{k\geq 1}$ such that $X_t^{n_k(t),x}$ converges to some $Y_t$ strongly in $L^2(\Omega, \Fcal_t)$. Following the same ideas as in the proof of \cite[Proposition 3.8]{Bauer_StrongSolutionsOfMFSDEs} it can be shown that the subsequence $\lbrace n_k(t) \rbrace_{k\geq 1}$ can be chosen independent of $t\in [0,T]$. Moreover, the proof of \cite[Proposition 3.9]{Bauer_StrongSolutionsOfMFSDEs} can be readily extended to the multi-dimensional case which yields that $\lbrace X_t^{n_k,x} \rbrace_{k\geq 1}$ converges weakly in $L^2(\Omega, \Fcal_t)$ to the unique strong solution $\Xo_t^x$ of the SDE
	\begin{align}\label{eq:XoSDE}
		d\Xo_t^x = b\left(t,\Xo_t^x, \Pbb_{Y_t} \right) dt + dB_t,~ t\in [0,T],~\Xo_0^x = x \in \Rbb^d.
	\end{align}
	Due to uniqueness of the limit we get that $Y_t^x \stackrel{d}{=} \Xo_t^x$ for all $t \in [0,T]$. Consequently, SDE \eqref{eq:XoSDE} is identical to mean-field SDE \eqref{eq:MFSDEAdditive} and thus $\lbrace X_t^{n_k, x} \rbrace_{k\geq 1}$ converges strongly in $L^2(\Omega, \Fcal_t)$ to $X_t = Y_t = \Xo_t$ for every $t\in [0,T]$.
\end{proof}

\begin{remark}
	For the sake of readability we assume subsequently without loss of generality that for every $t\in [0,T]$ the whole sequence $\lbrace X_t^{n,x} \rbrace$ converges strongly in $L^2(\Omega, \Fcal_t)$ to $X_t^x$.
\end{remark}

\begin{lemma}\label{lem:boundFV}
	Let $(X_t^{n,x})_{t\in[0,T]}$, $n\geq 1$, be the unique strong solutions of mean-field SDEs \eqref{eq:approxSDE}. Then, for any compact subset $K\subset \Rbb^d$ and $p\geq 2$,
	\begin{align*}
		\sup_{n\geq 1} \sup_{t\in[0,T]} \esssup_{x\in K} \EpO{\nabla_x X_t^{n,x}}{p} \leq C,
	\end{align*}
	for some constant $C>0$.
\end{lemma}
\begin{proof}
	In the course of this proof we make use of representation \eqref{eq:RepFVMD}, namely
	\begin{align*}
		\nabla_x X_t^{n,x} = D_0X_t^{n,x} + \int_0^t D_r X_t^{n,x} \nabla_x b(r, y, \Pbb_{X_r^{n,x}})\big\vert_{y= X_r^{n,x}} dr, \quad n\geq 1.
	\end{align*}
	First note that due to \cite[Lemma 3.5]{MenoukeuMeyerBrandisNilssenProskeZhang_VariationalApproachToTheConstructionOfStrongSolutions} and the uniform boundedness of $b_n$ in $n\geq 1$, we have that
	\begin{align}\label{eq:boundMD}
		\sup_{n\geq 1} \sup_{s,t \in [0,T]} \sup_{x\in K} \EpO{D_s X_t^{n,x}}{p} \leq C_1 < \infty,
	\end{align}
	for some constant $C_1>0$. Moreover, we get that
	\begin{align*}
		&\EpO{\int_0^t \nabla_x b_n(r, y, \Pbb_{X_r^{n,x}})\big\vert_{y= X_r^{n,x}} dr}{2p} \\
		&\quad \lesssim \sum_{k=1}^d \sum_{j=1}^d \Ebb \left[ \left( \int_0^t \left\vert \partial_{x^{(k)}} b_n^{(j)}(r, y, \Pbb_{X_r^{n,x}}) \right\vert_{y= X_r^{n,x}} dr \right)^{2p} \right].
	\end{align*}
	Following the proof of \cite[Lemma 3.10]{Bauer_StrongSolutionsOfMFSDEs}, we get due to the assumption $( \mu \mapsto b_n(t,y, \mu) ) \in \Lip(\Pcal_1(\Rbb^d))$ for every $t\in [0,T]$ and $y\in \Rbb^d$ uniformly in $n\geq 1$ that
	\begin{align*}
		\Ebb \left[ \left( \int_0^t \left\vert \partial_{x^{(k)}} b_n^{(j)}(r, y, \Pbb_{X_r^{n,x}}) \right\vert_{y= X_r^{n,x}} dr \right)^{2p} \right] \lesssim 1 + \int_0^t \esssup_{x\in \overline{\conv(K)}} \Ebb\left[ \left\vert \partial_{x^{(k)}} X_r^{n,(j),x} \right\vert \right] dr.
	\end{align*}
	All things considered we get that
	\begin{align*}
		\esssup_{x\in \overline{\conv(K)}} \Ep{\nabla_x X_t^{n,x}}{p} \lesssim 1 + \int_0^t \esssup_{x\in \overline{\conv(K)}} \Ep{\nabla_x X_r^{n,x}}{p} dr.
	\end{align*}
	Here, $\overline{\conv(K)}$ is the closure of the convex hull of the set $K$. Noting that $t \mapsto \esssup_{x\in \overline{\conv(K)}} \EpO{\nabla_x X_t^{n,x}}{p}$ is integrable over $[0,T]$ and Borel measurable, cf. \cite[Lemma 3.10]{Bauer_StrongSolutionsOfMFSDEs} for more details, allows for the application of Jones' generalization of Grönwall's inequality \cite[Lemma 5]{Jones_FundamentalInequalities}, and thus we get that
	\begin{align*}
		\esssup_{x\in K} \Ep{\nabla_x X_t^{n,x}}{p} \leq \esssup_{x\in \overline{\conv(K)}} \Ep{\nabla_x X_t^{n,x}}{p} < \infty.
	\end{align*}
\end{proof}

\begin{proof}[Proof of \Cref{thm:Sobolev}]
	The proof is equivalent to the proof of \cite[Theorem 3.3]{Bauer_StrongSolutionsOfMFSDEs} but for the sake of completeness we present it in the following. Consider the unique strong solutions $\lbrace X^{n,x} \rbrace_{n\geq 1}$ of mean-field SDEs \eqref{eq:approxSDE} and the unique strong solution $X^x$ of mean-field SDE \eqref{eq:MFSDEAdditive}. Subsequently, we show that  $\left\lbrace X^{n,x} \right\rbrace_{n\geq 1}$ is weakly relatively compact in $L^2(\Omega,W^{1,2}(U))$ and then identify the weak limit $Y:=\lim_{k\to\infty} X^{n_k}$ in $L^2(\Omega,W^{1,2}(U))$ with $X^x$, where $\lbrace n_k \rbrace_{k\geq 1}$ is a suitable subsequence.\par
	Note first that due to \Cref{lem:boundDrift} and \Cref{lem:boundFV}
	\begin{align*}
		\sup_{n\geq 1} \sup_{t\in [0,T]} \Ebb\left[ \Vert X_t^{n,x} \Vert_{W^{1,2}(U)}^2 \right] < \infty,
	\end{align*}
	and therefore, $\lbrace X_t^{n,x} \rbrace_{n\geq 1}$ is weakly relatively compact in $L^2(\Omega,W^{1,2}(U))$, see e.g. \cite[Theorem 10.44]{Leoni}. Thus, there exists a subsequence $\lbrace n_k \rbrace_{k\geq 0}$, such that $X_t^{n_k,x}$ converges weakly to some $Y_t \in L^2(\Omega, W^{1,2}(U))$ as $k\to \infty$. 
	Define for every $t\in [0,T]$
	\begin{align*}
		\left\langle X_t^n, \phi \right\rangle := \int_U X_t^{n,x} \phi(x) dx,
	\end{align*}
	 for some arbitrary test function $\phi \in \Ccal_0^{\infty}(U)$ and denote by $\phi'$ its first derivative.
	Then we get by \Cref{lem:boundDrift} that for all measurable sets $A \in \Fcal$ and $t\in[0,T]$
\begin{align*}
	\Ebb \left[ \mathbbm{1}_A \langle X_t^n -X_t, \phi' \rangle \right] \leq \Vert \phi' \Vert_{L^2(U)} \vert U \vert^{\frac{1}{2}} \sup_{x\in\overline{U}}\Ebb \left[ \mathbbm{1}_A \Vert X_t^{n,x}-X_t^x \Vert^2 \right]^{\frac{1}{2}} < \infty,
\end{align*}
where $\overline{U}$ is the closure of $U$. Hence, we get by \Cref{prop:L2Convergence} that $$\lim_{n\to\infty} \Ebb \left[ \mathbbm{1}_A \langle X_t^{n}-X_t, \phi' \rangle \right] = 0,$$ and thus,
\begin{small}
\begin{align*}
	\Ebb[\mathbbm{1}_A \langle X_t, \phi'\rangle] = \lim_{k\to\infty} \Ebb[\mathbbm{1}_A \langle X_t^{n_k}, \phi' \rangle] = - \lim_{k\to\infty} \Ebb\left[\mathbbm{1}_A \left\langle \nabla_x X_t^{n_k},\phi \right\rangle\right] = - \Ebb\left[\mathbbm{1}_A \left\langle \nabla_x Y_t,\phi \right\rangle\right].
\end{align*}
\end{small}
Consequently,
\begin{align}\label{eq:multiWeakDerivativeExists}
	\Pbb\text{-a.s.} \quad \left\langle X_t, \phi' \right\rangle = - \left\langle \nabla_x  Y_t, \phi \right\rangle.
\end{align}
It is left to show as in \cite[Theorem 3.4]{MeyerBrandisBanosDuedahlProske_ComputingDeltas} that there exists a measurable set $\Omega_0 \subset \Omega$ with full measure such that $(x \mapsto X_t^x)$ has a weak derivative on the subset $\Omega_0$. In order to show this we choose a sequence $\lbrace \phi_n \rbrace_{n\geq 1} \subset \Ccal_0^{\infty}(\Rbb)$ which is dense in $W^{1,2}(U)$ and a measurable subset $\Omega_n \subset \Omega$ with full measure such that \eqref{eq:multiWeakDerivativeExists} if fulfilled on $\Omega_n$ where $\phi$ is replaced by $\phi_n$. Then $\Omega_0 := \bigcap_{n\geq 1} \Omega_n$ is a full measure set such that $(x \mapsto X_t^x)$ has a weak derivative on it.
\end{proof}

Closing the part on Sobolev differentiability we consider the function $x \mapsto b\left(t,y,\Pbb_{X_t^x}\right)$ and show that it is weakly differentiable. In \Cref{sec:Bismut} the weak derivative $\nabla_x b\left(t,y,\Pbb_{X_t^x}\right)$ is used in the Bismut-Elworthy-Li formula. Further, we give a remark on the connection to the Lions derivative.

\begin{proposition}\label{prop:driftInSobolev}
	Suppose the drift coefficient $b:[0,T] \times \Rbb^d \times \Pcal_1(\Rbb^d) \to \Rbb^d$ is Lipschitz continuous in the third variable \eqref{eq:Lipschitz} and bounded. Let $(X_t^x)_{t \in [0,T]}$ be the unique strong solution of mean-field SDE \eqref{eq:MFSDEAdditive} and $U\subset \Rbb^d$ be an open and bounded subset. Then for every $1<p<\infty$, $t\in[0,T]$, and $y\in \Rbb^d$,
\begin{align*}
	\left(x \mapsto b\left(t,y,\Pbb_{X_t^x}\right) \right) \in W^{1,p}(U).
\end{align*}
\end{proposition}
\begin{proof}
	Using the proof of \Cref{lem:boundFV} the result follows equivalently to \cite[Proposition 3.11]{Bauer_StrongSolutionsOfMFSDEs}. Nevertheless for completeness we give the proof here. \par
	Consider the approximating sequence $\lbrace b_n \rbrace_{n\geq 1}$ of the drift function $b$ as defined in \eqref{eq:approxDrift} and let $(X_t^{n,x})_{t\in[0,T]}$, $n\geq 1$, be the corresponding unique strong solutions of mean-field SDEs \eqref{eq:approxSDE}. For the sake of readability we denote $b_n(x) := b_n\left(t,y,\Pbb_{X_t^{n,x}}\right)$ for every $n\geq 0$. First note that $\lbrace b_n \rbrace_{n\geq 1}$ is weakly relatively compact in $W^{1,p}(U)$, since by \Cref{lem:boundDrift} and the proof of \Cref{lem:boundFV}
	\begin{align*}
		\sup_{n\geq 1} \Vert b_n \Vert_{W^{1,p}(U)} < \infty,
	\end{align*}
	and thus the sequence is weakly relatively compact by \cite[Theorem 10.44]{Leoni}. Thus, there exists a subsequence $\lbrace n_k \rbrace_{k\geq 1}$ and $g\in W^{1,p}(U)$ such that $b_{n_k}$ converges weakly to $g$ as $k\to\infty$.\par
	Let $\phi \in \Ccal_0^{\infty}(U)$ be an arbitrary test-function with first derivative $\phi'$. Define
	\begin{align*}
		\left\langle b_n, \phi \right\rangle := \int_U b_n(x) \phi(x) dx.
	\end{align*}
	We get due to \Cref{lem:boundDrift} that
	\begin{align*}
		\langle b_n - b , \phi' \rangle \leq \Vert \phi' \Vert_{L^p(U)} \vert U \vert^{\frac{1}{p}} \sup_{x\in\overline{U}} \Vert b_n(x) - b(x) \Vert < \infty.
	\end{align*}
	Here, $\overline{U}$ is the closure of $U$. Further, by \Cref{prop:L2Convergence}
	\begin{align}\label{eq:convergenceDrift}
		&\left\Vert b_n\left(t,y,\Pbb_{X_t^{n,x}} \right) - b\left(t,y,\Pbb_{X_t^x} \right) \right\Vert \\
		&\quad \leq \left\Vert b_n\left(t,y,\Pbb_{X_t^{n,x}} \right) - b_n\left(t,y,\Pbb_{X_t^x} \right) \right\Vert + \left\Vert b_n\left(t,y,\Pbb_{X_t^x} \right) - b\left(t,y,\Pbb_{X_t^x} \right) \right\Vert \notag \\
		&\quad \leq C \Kcal\left(\Pbb_{X_t^{n,x}}, \Pbb_{X_t^x} \right) + \left\Vert b_n\left(t,y,\Pbb_{X_t^x} \right) - b\left(t,y,\Pbb_{X_t^x} \right) \right\Vert \xrightarrow[n\to\infty]{} 0, \notag
	\end{align}
	which yields $\lim_{n\to\infty} \langle b_n-b,\phi' \rangle = 0.$ Therefore,
\begin{align*}
	\langle b, \phi'\rangle = \lim_{k\to\infty} \langle b_{n_k}, \phi' \rangle = - \lim_{k\to\infty} \left\langle b_{n_k}',\phi \right\rangle = - \left\langle g',\phi \right\rangle,
\end{align*}
where $b_{n_k}'$ and $g'$ are the first variation processes of $b_{n_k}$ and $g$, respectively.	
\end{proof}

\begin{remark}
	Note that by the proof of \Cref{prop:driftInSobolev} the process $\nabla_x b$ is bounded, i.e.
	\begin{align}\label{eq:boundDb}
		\Vert \nabla_x b \Vert_\infty \leq C < \infty,
	\end{align}
	for some constant $C> 0$.
\end{remark}

\begin{remark}
	Due to \Cref{lem:boundDrift} the law of the unique strong solution $X^x$ of mean-field SDE \eqref{eq:MFSDEAdditive} is in the space $\Pcal_2(\Rbb^d)$ of probability measures with finite second moment. Thus, restraining the domain of the drift function $b$ to $[0,T] \times \Rbb^d \times \Pcal_2(\Rbb^d)$ enables the introduction of the Lions derivative $\nabla_\mu b(t,y,\cdot)$ for every $t\in [0,T]$ and $y \in \Rbb^d$. For an introduction to this topic we refer the reader to \cite{Cardaliaguet_NotesOnMeanFieldGames}. The analysis in \cite{BuckdahnLiPengRainer_MFSDEandAssociatedPDE} and \cite{crisan2018smoothing} of the first variation process $\nabla_x X_t^x$ suggests that the representation
	\begin{align*}
		\nabla_x b\left(t,y,\Pbb_{X_t^x}\right) = \EW{ \nabla_\mu b\left(t,y,\Pbb_{X_t^x}\right)(X_t^x) \nabla_x X_t^x}
	\end{align*}
	holds. Note that the Lions derivative entails an additional variable which is here denoted by $\nabla_\mu b\left( \cdot \right)(X_t^x)$.
\end{remark}

\subsection{Hölder continuity}\label{sec:Hölder}
	Concluding the section on regularity properties, we show Hölder continuity in time and space of the unique strong solution $(X_t^x)_{t\in[0,T]}$ of mean-field SDE \eqref{eq:MFSDEAdditive}.
	
\begin{theorem}\label{thm:Hoelder}
	Suppose the drift coefficient $b:[0,T] \times \Rbb^d \times \Pcal_1(\Rbb^d) \to \Rbb^d$ is Lipschitz continuous in the third variable \eqref{eq:Lipschitz} and bounded. Let $(X_t^x)_{t \in [0,T]}$ be the unique strong solution of mean-field SDE \eqref{eq:MFSDEAdditive}. Then, for every compact subset $K \subset \Rbb^d$ there exists a constant $C>0$ such that for all $s,t \in [0,T]$ and $x,y \in K$,
	\begin{align*}
		\Ebb\left[ \left\Vert X_t^x - X_s^y \right\Vert^2 \right] \leq C \left( \vert t-s \vert + \Vert x-y \Vert^2 \right).
	\end{align*}
	In particular, there exists a continuous version of the random field $(t,x) \mapsto X_t^x$ with Hölder continuous trajectories of order $\alpha< \frac{1}{2}$ in $t \in [0,T]$ and $\alpha<1$ in $x\in \Rbb^d$.
\end{theorem}

	The proof of \Cref{thm:Hoelder} is analogous to the proof of \cite[Theorem 3.12]{Bauer_StrongSolutionsOfMFSDEs} and uses the following lemma.
	
\begin{lemma}
	Suppose the drift coefficient $b:[0,T] \times \Rbb^d \times \Pcal_1(\Rbb^d) \to \Rbb^d$ is Lipschitz continuous in the third variable \eqref{eq:Lipschitz} and bounded. Let $(X_t^x)_{t \in [0,T]}$ be the unique strong solution of mean-field SDE \eqref{eq:MFSDEAdditive}. Then, for every compact subset $K \subset \Rbb^d$ and $p\geq 1$, there exists a constant $C>0$ such that
	\begin{align*}
		\sup_{t\in[0,T]} \esssup_{x\in K} \Ebb\left[ \left\Vert \nabla_x X_t^x \right\Vert^p \right] \leq C.
	\end{align*}
\end{lemma}
\begin{proof}
	The result follows immediately by \Cref{lem:boundFV} and Fatou's lemma.
\end{proof}

\section{Bismut-Elworthy-Li type formula}\label{sec:Bismut}
	In this section we establish an integration by parts formula of Bismut-Elworthy-Li type. More precisely, we consider the functional $x \mapsto \Ebb\left[ \Phi(X_t^x) \right]$, where $\Phi$ merely fulfills some integrability condition, and show that it is weakly differentiable. Moreover, we give a probabilistic representation of the derivative $\nabla_x \Ebb\left[ \Phi(X_t^x) \right]$.

\begin{theorem}\label{thm:Bismut}
	Suppose the drift coefficient $b:[0,T] \times \Rbb^d \times \Pcal_1(\Rbb^d) \to \Rbb^d$ is Lipschitz continuous in the third variable \eqref{eq:Lipschitz} and bounded. Let $(X_t^x)_{t \in [0,T]}$ be the unique strong solution of mean-field SDE \eqref{eq:MFSDEAdditive}, $K\subset \Rbb^d$ be a compact subset, $\Phi \in L^{2}(\Rbb^d;\omega_T)$, and $\omega_T$ is as defined in \eqref{eq:weightFunction}. Then, for every open subset $U \subset K$, $t\in[0,T]$, and $1<q<\infty$,
	\begin{align*}
		\left(x \mapsto \Ebb\left[ \Phi(X_t^x) \right] \right) \in W^{1,q}(U),
	\end{align*}
	and for almost all $x \in K$
	\begin{small}	
	\begin{align}\label{eq:Delta}
		\nabla_x \Ebb[\Phi(X_T^x)] = \Ebb \left[ \Phi(X_T^x)  \int_0^T \left( a(s) \nabla_x X_s^x + \nabla_x b\left(s,y,\Pbb_{X_s^x}\right)\vert_{y=X_s^x} \int_0^s a(u) du \right) dB_s \right],
	\end{align}
	\end{small}
	where $a: \Rbb  \to \Rbb$ is any bounded, measurable function such that
	\begin{align*}
		\int_0^T a(s) ds = 1.
	\end{align*}
\end{theorem}
\begin{proof}
	First assume that $\Phi \in \Ccal_b^{1,1}(\Rbb^d)$ and let $\lbrace b_n \rbrace_{n\geq 1}$ and $\lbrace X^{n,x} \rbrace_{n\geq 1}$ be defined as in \eqref{eq:approxDrift} and \eqref{eq:approxSDE}, respectively. Due to \Cref{prop:L2Convergence} it is readily seen that
	\begin{align}\label{eq:convergenceEW}
		\EW{\Phi(X_T^{n,x})} \xrightarrow[n\to \infty]{} \EW{\Phi(X_T^x)},
	\end{align}
	for every $t\in [0,T]$ and $x \in K$. Equivalently to \cite[Lemma 4.1]{Bauer_StrongSolutionsOfMFSDEs} it can be shown that $\EW{\Phi(X_t^{n,x})}$ is weakly differentiable in $x$ and
	\begin{align*}
		\nabla_x \EW{\Phi(X_T^{n,x})} = \EW{\Phi'(X_T^{n,x}) \nabla_x X_T^{n,x}}.
	\end{align*}
	Furthermore, using the representation \eqref{eq:RepFVMD} we get for any bounded measurable function $a: \Rbb  \to \Rbb$ with $\int_0^T a(s) ds = 1$ that
	\begin{small}
	\begin{align*}
		\nabla_x X_T^{n,x} &= \int_0^T a(s) \left( D_sX_T^{n,x} \nabla_x X_s^{n,x} + \int_s^T D_r X_T^{n,x} \nabla_x b_n(r, y, \Pbb_{X_r^{n,x}})\big\vert_{y= X_r^{n,x}} dr \right) ds \\
		&= \int_0^T a(s) D_sX_T^{n,x} \nabla_x X_s^{n,x} ds + \int_0^T \int_s^T a(s) D_r X_T^{n,x} \nabla_x b_n(r, y, \Pbb_{X_r^{n,x}})\big\vert_{y= X_r^{n,x}} dr ds.
	\end{align*}
	\end{small}
	Now we look at each term individually starting by the first one. Note first that $\Phi(X_T^{n,x})$ is Malliavin differentiable and thus using the chain rule yields
	\begin{align*}
		\EW{ \Phi'(X_T^{n,x}) \int_0^T a(s) D_sX_T^{n,x} \nabla_x X_s^{n,x} ds} = \EW{\int_0^T a(s) D_s \Phi(X_T^{n,x}) \nabla_x X_s^{n,x} ds}.
	\end{align*}
	Since $s \mapsto a(s) \nabla_x X_s^{n,x}$ is an adapted process and by \Cref{lem:boundFV}
	\begin{align*}
		\EW{\int_0^T \left\Vert a(s) \nabla_x X_s^{n,x} \right\Vert^2 ds} < \infty,
	\end{align*}
	the application of the duality formula \cite[Corollary 4.4]{ProskeDiNunnoOksendal_MalliavinCalculus} yields
	\begin{align*}
		\EW{\int_0^T a(s) D_s \Phi(X_T^{n,x}) \nabla_x X_s^{n,x} ds} = \EW{\Phi(X_T^{n,x}) \int_0^T a(s) \nabla_x X_s^{n,x} dB_s}.
	\end{align*}
	Considering the second term note first that due to \eqref{eq:boundMD} and \eqref{eq:boundDb}
	\begin{align*}
		\sup_{r, s \in [0,T]} \Eabs{\Phi'(X_T^{n,x}) a(s) D_r X_T^{n,x} \nabla_x b_n(r, y, \Pbb_{X_r^{n,x}})\big\vert_{y= X_r^{n,x}}} < \infty.
	\end{align*}
	Consequently, the integral 
	\begin{align*}
		\int_0^T \int_0^T \EW{\Phi'(X_T^{n,x}) a(s) D_r X_T^{n,x} \nabla_x b_n(r, y, \Pbb_{X_r^{n,x}})\big\vert_{y= X_r^{n,x}}} dr ds
	\end{align*}
	exists and is finite by Tonelli's Theorem. Thus, the order of integration can be swapped and we obtain by using once more the duality formula \cite[Corollary 4.4]{ProskeDiNunnoOksendal_MalliavinCalculus} that
	\begin{align*}
		&\EW{\Phi'(X_T^{n,x}) \int_0^T \int_s^T a(s) D_r X_T^{n,x} \nabla_x b_n(r, y, \Pbb_{X_r^{n,x}})\big\vert_{y= X_r^{n,x}} dr ds} \\
		&\quad =\EW{\int_0^T  D_r \Phi(X_T^{n,x})  \nabla_x b_n(r, y, \Pbb_{X_r^{n,x}})\big\vert_{y= X_r^{n,x}} \int_0^r a(s) ds dr} \\
		&\quad =\EW{\Phi(X_T^{n,x}) \int_0^T  \nabla_x b_n(r, y, \Pbb_{X_r^{n,x}})\big\vert_{y= X_r^{n,x}} \int_0^r a(s) ds dB_r}.
	\end{align*}
	Putting all together we obtain representation \eqref{eq:Delta} for $\Phi \in \Ccal_b^{1,1}(\Rbb^d)$ where $b$ and $X^x$ are substituted by $b_n$ and $X^{n,x}$, respectively. \par
	Next, we show that representation \eqref{eq:Delta} is valid also for $b$ and $X^x$. Let $\varphi \in \Ccal_0^\infty(U)$. We prove subsequently that
	\begin{small}
	\begin{align*}
		&\int_U \varphi'(x) \EW{\Phi(X_T^x)} dx \\
		&\quad = - \int_U \varphi(x) \Ebb \left[ \Phi(X_T^x)  \int_0^T \left( a(s) \nabla_x X_s^x + \nabla_x b\left(s,y,\Pbb_{X_s^x}\right)\vert_{y=X_s^x} \int_0^s a(u) du \right) dB_s \right] dx.
	\end{align*}
	\end{small}
	Using \eqref{eq:convergenceEW} we have that
	\begin{align*}
		&\int_U \varphi'(x) \EW{\Phi(X_T^x)} dx \\
		&\quad = - \lim_{n\to \infty} \int_U \varphi(x) \Ebb \left[ \Phi(X_T^{n,x})  \int_0^T \left( a(s) \nabla_x X_s^{n,x} + \nabla_x b_n(s,x) \int_0^s a(u) du \right) dB_s \right] dx \\
		&\quad = - \lim_{n\to \infty} \int_U \varphi(x) \Ebb \left[ \Phi(X_T^{n,x})  \int_0^T a(s) \nabla_x X_s^{n,x} dB_s \right] dx \\
		&\qquad - \lim_{n\to \infty} \int_U \varphi(x) \Ebb \left[ \Phi(X_T^{n,x})  \int_0^T \nabla_x b_n(s,x) \int_0^s a(u) du dB_s \right] dx \\
		&\quad =: - \lim_{n\to \infty} A_n - \lim_{n\to \infty} C_n,
	\end{align*}
	where $b_n(s,x) := b_n\left(s,y,\Pbb_{X_s^{n,x}}\right)\vert_{y=X_s^{n,x}}$, $n\geq 0$. For $A_n$ we further get that
	\begin{align*}
	A_n &= \int_U \varphi(x) \Ebb \left[ \left( \Phi(X_T^{n,x}) - \Phi(X_T^x) \right) \int_0^T a(s) \nabla_x X_s^{n,x} dB_s \right] dx \\
	&\quad + \int_U \varphi(x) \Ebb \left[ \Phi(X_T^x)  \int_0^T a(s) \left( \nabla_x X_s^{n,x} - \nabla_x X_s^x \right) dB_s \right] dx \\
	&\quad + \int_U \varphi(x) \Ebb \left[ \Phi(X_T^x)  \int_0^T a(s) \nabla_x X_s^x dB_s \right] dx \\
	&=: A_n(I) + A_n(II) + \int_U \varphi(x) \Ebb \left[ \Phi(X_T^x)  \int_0^T a(s) \nabla_x X_s^x dB_s \right] dx.
	\end{align*}
	Note that $A_n(I)$ and $A_n(II)$ converge to $0$ due to \Cref{prop:L2Convergence} and \Cref{lem:boundFV}, and the proof of \Cref{thm:Sobolev}, respectively.\par 
	For $B_n$ let us first define the measure change 
	\begin{align*}
		\frac{d \Qbb^n}{d \Pbb} := \Ecal\left( - \int_0^T b_n\left( s, X_s^{n,x}, \Pbb_{X_s^{n,x}} \right) dB_s \right), \quad n\geq 0.
	\end{align*}		
	Note that under $\Qbb^n$ the processes $X^{n,x}$ is Brownian motion. Hence, we get with
	\begin{align*}
		\Ecal_T^n := \Ecal\left( \int_0^T b_n\left( s, B_s^x, \Pbb_{X_s^{n,x}} \right) dB_s \right), \quad n\geq 0,
	\end{align*}
	that
	\begin{small}
	\begin{align*}
		&C_n - C_0 = \int_U \varphi(x) \left( \Ebb \left[ \Phi(B_T^x)  \int_0^T \nabla_x b_n(s,x) \int_0^s a(u) du dB_s ~ \Ecal_T^n \right] \right. \\
		&\quad - \left. \Ebb \left[ \Phi(B_T^x)  \int_0^T \nabla_x b_0(s,x) \int_0^s a(u) du dB_s ~\Ecal_T^0 \right] \right) dx\\
		&\quad - \int_U \varphi(x) \left( \Ebb \left[ \Phi(B_T^x)  \int_0^T \nabla_x b_n(s,x) \int_0^s a(u) du ~ b_n\left(s,B_s^x, \Pbb_{X_s^{n,x}} \right) ds ~ \Ecal_T^n \right] \right. \\
		&\quad - \left. \Ebb \left[ \Phi(B_T^x)  \int_0^T \nabla_x b_0(s,x) \int_0^s a(u) du ~ b\left(s,B_s^x, \Pbb_{X_s^x} \right) ds ~\Ecal_T^0 \right] \right) dx \\
		&=: C_n(I) - C_n(II),
	\end{align*}
	\end{small}
	where $b_n(s,x) := b_n\left(s,y,\Pbb_{X_s^{n,x}}\right)\vert_{y=B_s^x}$, $n\geq 0$. Considering $C_n(I)$ we have due to \eqref{eq:boundDb}
	\begin{align*}
		C_n(I) &\lesssim \int_U \varphi(x) \left( \Ebb \left[ \int_0^T \left\vert \nabla_x b_n\left(s,y,\Pbb_{X_s^{n,x}}\right) - \nabla_x b\left(s,y,\Pbb_{X_s^x} \right) \right\vert_{y=B_s^x}^2 ds \right]^\frac{1}{2} \right. \\
		&\quad + \left. \Ebb \left[ \left\vert \Ecal_T^n - \Ecal_T^0 \right\vert^2 \right] \right) dx.
	\end{align*}
	The first term converges to $0$ due to the proof of \Cref{prop:driftInSobolev} whereas the second term converges to $0$ due to \Cref{lem:convergenceEcal}. Furthermore, for $C_n(II)$ we have
	\begin{small}
	\begin{align*}
		C_n(II) &\lesssim \int_U \varphi(x) \left( C_n(I) + \Ebb \left[ \Phi(B_T^x)  \int_0^T \left\vert b_n\left(s,B_s^x, \Pbb_{X_s^{n,x}} \right) - b\left(s,B_s^x, \Pbb_{X_s^x} \right) \right\vert ds \right] \right) dx,
	\end{align*}
	\end{small}
	which converges to $0$ due to \eqref{eq:convergenceDrift} and dominated convergence. Thus equation \eqref{eq:Delta} holds for $\Phi \in \Ccal_b^{1,1}(\Rbb^d)$. \par
	Lastly, we show that equation \eqref{eq:Delta} holds true for $\Phi \in L^2(\Rbb^d; \omega_T)$. In order to show this, define a sequence $\lbrace \Phi_n \rbrace \subset \Ccal_b^{1,1}(\Rbb^d)$ by standard arguments which approximates $\Phi$ with respect to the norm $L^2(\Rbb^d; \omega_T)$. 
	Note first that
	\begin{align}\label{eq:MultiWellDefined}
		&\Eabs{\Phi(X_T^x)  \int_0^T \left( a(s) \nabla_x X_s^x + \nabla_x b\left(s,y,\Pbb_{X_s^x}\right)\vert_{y=X_s^x} \int_0^s a(u) du \right) dB_s } \\
		&\quad \leq \Ep{\Phi(X_T^x)}{2} \notag\\
		&\qquad \times \Ep{\int_0^T \left( a(s) \nabla_x X_s^x + \nabla_x b(s,X_s^x,\Pbb_{X_s^x})\vert_{y=X_s^x} \int_0^s a(u) du \right) dB_s}{2} \notag \\
		&\quad \leq \Ebb \left[ \left\Vert \Phi(B_T^x) \right\Vert^2 \Ecal\left( \int_0^T b(s,B_s^x,\Pbb_{X_s^x}) dB_s \right) \right]^{\frac{1}{2}} \notag\\
		&\qquad \times \Ebb \left[ \int_0^T \left\Vert a(s) \nabla_x X_s^x + \nabla_x b(s,X_s^x,\Pbb_{X_s^x})\vert_{y=X_s^x} \int_0^s a(u) du \right\Vert^2 du \right]^{\frac{1}{2}} \notag\\
		&\quad \lesssim \Ep{\Phi(B_T^x)}{2} < \infty, \notag
	\end{align}
	where we have used \Cref{lem:boundFV} and \eqref{eq:boundDb}. Thus, expression \eqref{eq:MultiWellDefined} is well-defined. Furthermore, it is readily seen that 
	\begin{align*}
		\EW{\Phi_n(X_T^x)} \xrightarrow[n\to \infty]{} \EW{\Phi(X_T^x)}.
	\end{align*}
	Thus, for any test function $\varphi \in \Ccal_0^\infty(\Rbb^d)$ we have that
	\begin{align*}
		&\int_U \varphi(x) \EW{\Phi(X_T^x)} dx \\
		&\quad = - \lim_{n\to \infty} \int_U \varphi'(x) \Ebb \left[ \Phi_n(X_T^x) \int_0^T \left( a(s) \nabla_x X_s^x + \nabla_x b(s,x) \int_0^s a(u) du \right) dB_s \right] dx \\
		&\quad \lesssim - \lim_{n\to \infty} \int_U \varphi'(x) \Ebb \left[ \left( \Phi_n(X_T^x) - \Phi(X_T^x) \right)^2 \right]^\frac{1}{2} dx \\
		&\qquad - \int_U \varphi'(x) \Ebb \left[ \Phi(X_T^x) \int_0^T \left( a(s) \nabla_x X_s^x + \nabla_x b(s,x) \int_0^s a(u) du \right) dB_s \right] dx \\
		&\quad = - \int_U \varphi'(x) \Ebb \left[ \Phi(X_T^x) \int_0^T \left( a(s) \nabla_x X_s^x + \nabla_x b(s,x) \int_0^s a(u) du \right) dB_s \right] dx,
	\end{align*}
	where $b(s,x) := b(s,y,\Pbb_{X_s^x})\vert_{y=X_s^x}$. Consequently, equation \eqref{eq:Delta} holds for $\Phi \in L^2(\Rbb^d; \omega_T)$.
\end{proof}

\appendix
\section{Technical Results}

Consider the (mean-field) stochastic differential equation
\begin{align}\label{eq:lawSDE}
	dX_t^{x,\mu} = b \left(t,X_t^{x,\mu}, \mu_t \right)dt + dB_t, ~t\in [0,T], ~X_0^{x,\mu} = x \in \Rbb^d,
\end{align}
where $\mu \in \Ccal([0,T]; \Pcal_1(\Rbb^d))$. The following lemmas can be proven similar to \cite[Lemma A.1 \& Lemma A.6]{Bauer_StrongSolutionsOfMFSDEs}.

\begin{lemma}\label{lem:boundDrift}
	Suppose the drift coefficient $b:[0,T] \times \Rbb^d \times \Pcal_1(\Rbb^d) \to \Rbb^d$ is of at most linear growth \eqref{eq:linearGrowth} and $X^{x,\mu}$ is a solution of SDE \eqref{eq:lawSDE}. Then, for every $p\geq 1$ and any compact subset $K\subset \Rbb^d$
	\begin{align*}
		\sup_{x\in K} \Ebb \left[ \sup_{t\in [0,T]}  \Vert b\left(t,X_t^{x,\mu}, \mu_t \right) \Vert^p \right] < \infty.
	\end{align*}
	In particular,
	\begin{align*}
		\sup_{x\in K} \Ebb \left[ \sup_{t\in [0,T]} \Vert X_t^{x,\mu} \Vert^p \right] < \infty.
	\end{align*}
	Moreover, for a set of measures $E_C := \lbrace \mu \in \Ccal([0,T]; \Pcal_1(\Rbb)) : \sup_{t\in [0,T]} \Kcal(\mu_t, \delta_0) \leq C \rbrace$, where $C>0$ is some constant, and every $p\geq 1$
	\begin{align*}
		\sup_{x\in K} \Ebb \left[ \sup_{t\in [0,T]} \sup_{\mu \in E_C} \Vert b\left(t,X_t^{x,\mu}, \mu_t \right) \Vert^p \right] < \infty.
	\end{align*}
\end{lemma}

\begin{lemma}\label{lem:convergenceEcal}
	Suppose the drift coefficient $b:[0,T] \times \Rbb^d \times \Pcal_1(\Rbb^d) \to \Rbb^d$ is Lipschitz continuous in the third variable \eqref{eq:Lipschitz} and bounded. Let $(X_t^x)_{t \in [0,T]}$ be the unique strong solution of mean-field SDE \eqref{eq:MFSDEAdditive}. Furthermore, $\lbrace b_n \rbrace_{n\geq 1}$ is the approximating sequence of $b$ as defined in \eqref{eq:approxDrift} and $(X_t^{n,x})_{t\in[0,T]}$, $n\geq 1$, the corresponding unique strong solutions of mean-field SDEs \eqref{eq:approxSDE}. Then for any $p\geq 1$
	\begin{align*}
		\Ebb \left[ \left| \Ecal \left( \int_0^T b_n(t,B_t^x,\Pbb_{X_t^{n,x}}) dB_t \right) - \Ecal \left( \int_0^T b(t,B_t^x,\Pbb_{X_t^x}) dB_t \right) \right|^p \right]^{\frac{1}{p}} \xrightarrow[n\to\infty]{} 0.
	\end{align*}
\end{lemma}

\begin{footnotesize}
	\bibliography{literatureTH}

\begin{thebibliography}{10}

\bibitem{MeyerBrandisBanosDuedahlProske_ComputingDeltas}
D.~Ba{\~n}os, T.~Meyer-Brandis, F.~Proske, and S.~Duedahl.
\newblock {Computing Deltas Without Derivatives}.
\newblock {\em Finance and Stochastics}, 21(2):509--549, 2017.

\bibitem{Bauer_StrongSolutionsOfMFSDEs}
M.~Bauer, T.~Meyer-Brandis, and F.~Proske.
\newblock {Strong solutions of mean-field stochastic differential equations
  with irregular drift}.
\newblock {\em Electronic Journal of Probability}, 23, 2018.

\bibitem{BuckdahnDjehicheLiPeng_MFBSDELimitApproach}
R.~Buckdahn, B.~Djehiche, J.~Li, and S.~Peng.
\newblock {Mean-field backward stochastic differential equations: a limit
  approach}.
\newblock {\em The Annals of Probability}, 37(4):1524--1565, 2009.

\bibitem{BuckdahnLiPeng_MFBSDEandRelatedPDEs}
R.~Buckdahn, J.~Li, and S.~Peng.
\newblock {Mean-field backward stochastic differential equations and related
  partial differential equations}.
\newblock {\em Stochastic Processes and their Applications},
  119(10):3133--3154, 2009.

\bibitem{BuckdahnLiPengRainer_MFSDEandAssociatedPDE}
R.~Buckdahn, J.~Li, S.~Peng, and C.~Rainer.
\newblock {Mean-field stochastic differential equations and associated PDEs}.
\newblock {\em The Annals of Probability}, 45(2):824--878, 2017.

\bibitem{Cardaliaguet_NotesOnMeanFieldGames}
P.~Cardaliaguet.
\newblock {Notes on Mean Field Games (from P.-L. Lions' lectures at Coll{\`e}ge
  de France}.
\newblock {\em Available on the website of Coll{\`e}ge de France
  (https://www.ceremade.dauphine.fr/~cardalia/MFG100629.pdf)}, 2013.

\bibitem{CarmonaDelarue_ProbabilisticAnalysisofMFG}
R.~Carmona and F.~Delarue.
\newblock {Probabilistic Analysis of Mean-Field Games.}
\newblock {\em SIAM Journal on Control and Optimization}, 51(4):2705--2734,
  2013.

\bibitem{CarmonaDelarue_MasterEquation}
R.~Carmona and F.~Delarue.
\newblock {The master equation for large population equilibriums}.
\newblock {\em {Stochastic Analysis and Applications 2014}}, pages 77--128,
  2014.

\bibitem{CarmonaDelarue_FBSDEsandControlledMKV}
R.~Carmona and F.~Delarue.
\newblock {Forward--backward stochastic differential equations and controlled
  McKean--Vlasov dynamics}.
\newblock {\em The Annals of Probability}, 43(5):2647--2700, 09 2015.

\bibitem{CarmonaDelarue_Book}
R.~Carmona and F.~Delarue.
\newblock {\em {Probabilistic Theory of Mean Field Games with Applications
  I-II}}.
\newblock Springer, 2018.

\bibitem{CarmonaDelarueLachapelle_ControlofMKVvsMFG}
R.~Carmona, F.~Delarue, and A.~Lachapelle.
\newblock {Control of McKean--Vlasov dynamics versus mean field games}.
\newblock {\em Mathematics and Financial Economics}, 7(2):131--166, 2013.

\bibitem{CarmonaFouqueMousaviSun_SystemicRiskandStochasticGameswithDelay}
R.~Carmona, J.-P. Fouque, S.~M. Mousavi, and L.-H. Sun.
\newblock {Systemic Risk and Stochastic Games with Delay}.
\newblock {\em Journal of Optimization Theory and Applications},
  179(2):366--399, 2018.

\bibitem{CarmonaFouqueSun_MFGandSystemicRisk}
R.~Carmona, J.-P. Fouque, and L.-H. Sun.
\newblock {Mean field games and systemic risk}.
\newblock {\em Communications in Mathematical Sciences}, 13(4):911--933, 2015.

\bibitem{CarmonaLacker_ProbabilisticWeakFormulationofMFGandApplications}
R.~Carmona and D.~Lacker.
\newblock {A probabilistic weak formulation of mean field games and
  applications}.
\newblock {\em The Annals of Applied Probability}, 25(3):1189--1231, 2015.

\bibitem{Chiang_MKVWithDiscontinuousCoefficients}
T.~Chiang.
\newblock {McKean-Vlasov equations with discontinuous coefficients}.
\newblock {\em Soochow J. Math}, 20(4):507--526, 1994.

\bibitem{crisan2018smoothing}
D.~Crisan and E.~McMurray.
\newblock {Smoothing properties of McKean--Vlasov SDEs}.
\newblock {\em Probability Theory and Related Fields}, 171(1-2):97--148, 2018.

\bibitem{de2015strong}
P.~C. de~Raynal.
\newblock {Strong well posedness of McKean-Vlasov stochastic differential
  equations with H{\"o}lder drift}.
\newblock {\em Stochastic Processes and their Applications}, 130(1):79 -- 107,
  2020.

\bibitem{ProskeDiNunnoOksendal_MalliavinCalculus}
G.~{Di Nunno}, B.~{\O}ksendal, and F.~Proske.
\newblock {\em {Malliavin Calculus for L{\'e}vy Processes with Applications to
  Finance}}.
\newblock {Universitext}. Springer-Verlag Berlin Heidelberg, first edition,
  2009.

\bibitem{FouqueIchiba_StabilityinaModelofInterbankLending}
J.-P. Fouque and T.~Ichiba.
\newblock {Stability in a model of inter-bank lending}.
\newblock {\em SIAM Journal on Financial Mathematics}, 4:784--803, 2013.

\bibitem{FouqueSun_SystemicRiskIllustrated}
J.-P. Fouque and L.-H. Sun.
\newblock {Systemic risk illustrated}.
\newblock {\em Handbook on Systemic Risk}, pages 444--452, 2013.

\bibitem{GarnierPapanicolaouYang_LargeDeviationsforMFModelofSystemicRisk}
J.~Garnier, G.~Papanicolaou, and T.-W. Yang.
\newblock {Large deviations for a mean field model of systemic risk}.
\newblock {\em SIAM Journal on Financial Mathematics}, 4(1):151--184, 2013.

\bibitem{Jones_FundamentalInequalities}
G.~S. Jones.
\newblock {Fundamental inequalities for discrete and discontinuous functional
  equations}.
\newblock {\em Journal of the Society for Industrial and Applied Mathematics},
  12(1):43--57, 1964.

\bibitem{JourdainMeleardWojbor_NonlinearSDEs}
B.~Jourdain, S.~M{\'e}l{\'e}ard, and W.~A. Woyczynski.
\newblock {Nonlinear SDEs driven by L{\'e}vy processes and related PDEs}.
\newblock {\em ALEA, Latin American Journal of Probability}, 4:1--29, 2008.

\bibitem{Kac_FoundationsOfKineticTheory}
M.~Kac.
\newblock {Foundations of Kinetic Theory}.
\newblock In {\em {Proceedings of the Third Berkeley Symposium on Mathematical
  Statistics and Probability, Volume 3: Contributions to Astronomy and
  Physics}}, pages 171--197, Berkeley, Calif., 1956. University of California
  Press.

\bibitem{KleyKlueppelbergReichel_SystemicRiskTroughContagioninaCorePeripheryStructuredBankingNetwork}
O.~Kley, C.~Kl{\"u}ppelberg, and L.~Reichel.
\newblock {Systemic risk through contagion in a core-periphery structured
  banking network}.
\newblock {\em Banach Center Publications}, 104(1):133--149, 2015.

\bibitem{krylov1969ito}
N.~Krylov.
\newblock {On Ito's stochastic integral equations}.
\newblock {\em Theory of Probability \& Its Applications}, 14(2):330--336,
  1969.

\bibitem{LasryLions_MeanFieldGames}
J.-M. Lasry and P.-L. Lions.
\newblock {Mean field games}.
\newblock {\em Japanese Journal of Mathematics}, 2(1):229--260, 2007.

\bibitem{Leoni}
G.~Leoni.
\newblock {\em {A first course in Sobolev spaces}}, volume 105.
\newblock American Mathematical Society Providence, RI, 2009.

\bibitem{LiMin_WeakSolutions}
J.~Li and H.~Min.
\newblock {Weak Solutions of Mean-Field Stochastic Differential Equations and
  Application to Zero-Sum Stochastic Differential Games}.
\newblock {\em SIAM Journal on Control and Optimization}, 54(3):1826--1858,
  2016.

\bibitem{mahmudov2006mckean}
N.~Mahmudov and M.~McKibben.
\newblock {McKean-Vlasov stochastic differential equations in Hilbert spaces
  under Caratheodory conditions}.
\newblock {\em Dynamic Systems and Applications}, 15(3/4):357, 2006.

\bibitem{McKean_AClassOfMarkovProcess}
H.~P. McKean.
\newblock {A class of Markov processes associated with nonlinear parabolic
  equations}.
\newblock {\em Proceedings of the National Academy of Sciences of the United
  States of America}, 56(6):1907--1911, 1966.

\bibitem{MenoukeuMeyerBrandisNilssenProskeZhang_VariationalApproachToTheConstructionOfStrongSolutions}
O.~Menoukeu-Pamen, T.~Meyer-Brandis, T.~Nilssen, F.~Proske, and T.~Zhang.
\newblock {A variational approach to the construction and Malliavin
  differentiability of strong solutions of SDE's}.
\newblock {\em Mathematische Annalen}, 357(2):761--799, 2013.

\bibitem{MishuraVeretennikov_SolutionsOfMKV}
Y.~Mishura and A.~Veretennikov.
\newblock {Existence and uniqueness theorems for solutions of McKean--Vlasov
  stochastic equations}.
\newblock {\em {arXiv preprint arXiv:1603.02212}}, 2016.

\bibitem{veretennikov1981strong}
A.~Veretennikov.
\newblock {On strong solutions and explicit formulas for solutions of
  stochastic integral equations}.
\newblock {\em Sbornik: Mathematics}, 39(3):387--403, 1981.

\bibitem{Vlasov_VibrationalPropertiesofElectronGas}
A.~Vlasov.
\newblock {The vibrational properties of an electron gas}.
\newblock {\em Soviet Physics Uspekhi}, 10(6):721, 1968.

\end{thebibliography}
	\bibliographystyle{abbrv}
\end{footnotesize}
\bigskip
\rule{\textwidth}{1pt}

\end{document}